\documentclass[11pt]{article}

\usepackage{amsfonts}
\usepackage{amsmath}

\newcommand{\beq}{\begin{equation}}
\newcommand{\eeq}{\end{equation}}
\newcommand{\bea}{\begin{eqnarray}}
\newcommand{\eea}{\end{eqnarray}}
\newcommand{\beas}{\begin{eqnarray*}}
\newcommand{\eeas}{\end{eqnarray*}}

\newtheorem{theorem}{Theorem}[section]

\newtheorem{proposition}[theorem]{Proposition}

\newtheorem{lemma}[theorem]{Lemma}
\newtheorem{remark}[theorem]{Remark}
\newtheorem{example}[theorem]{Example}
\newtheorem{examples}[theorem]{Examples}
\newtheorem{foo}[theorem]{Remarks}

\newenvironment{proof}{\addvspace{\medskipamount}\par\noindent{\it
Proof}.}
{\unskip\nobreak\hfill$\Box$\par\addvspace{\medskipamount}}

\title{Bakry-\'Emery meet Villani}
\author{Fabrice Baudoin\footnote{fabrice.baudoin@uconn.edu, Research partially funded by NSF-DMS 1511328 and NSF-DMS 1660031}}

\date{Department of Mathematics \\
 University of Connecticut
}

\begin{document}
\maketitle

\begin{abstract}
We revisit Villani's approach to the study of hypocoercive diffusion operators by applying a variant of the Bakry-\'Emery machinery. The method relies on a generalized Bakry-Emery type criterion that applies to Kolmogorov type operators. Our approach includes as a special case the kinetic Fokker-Planck equation  and allows, in that case,  to recover and somehow improve  hypocoercive estimates first obtained by Villani. 
 \end{abstract}

\tableofcontents

\section{Introduction}

We study gradient bounds and convergence to equilibrium for the semigroup generated by a diffusion operator of the form $L=\sum_{i=1}^n X_i^2 +Y$, where $X_1,\cdots,X_n,Y$ are vector fields. Though our methods are more general, we are particularly interested in the case where $L$ is hypoelliptic and $\sum_{i=1}^n X_i^2$ is not, that is the hypoellipticity comes from the first order operator $Y$. The problem of convergence to an equilibrium in this type of situation has attracted a lot of interest in the literature because evolution equations involving a degenerate dissipative operator and a conservative operator naturally arise in many fields of applied mathematics: We refer to Villani's memoir \cite{Villani1} and to the references therein for a discussion about this. 

\

There have been different approaches to tackle this problem. A functional analytic approach, based on previous ideas by Kohn and H\"ormander, uses pseudo-differential calculus and delicate spectral localization tools to prove exponential convergence to equilibrium with explicit bounds on the rate. For this approach, we refer to Eckmann and Hairer \cite{EH}, H\'erau and Nier \cite{HN1}, and  Heffer and Nier \cite{HN2}. 

\

Villani in his memoir \cite{Villani1} introduces the concept of hypocoercivity and derives very general sufficient conditions ensuring the convergence to an equilibrium. The main strategy, already implicit in the work by Talay \cite{talay} , is to work in a suitable Hilbert space associated to the equation and to find in this Hilbert space a nice norm which is equivalent to the original one, but with respect to which convergence to equilibrium is easy to obtain; We refer to Section 4.1 in \cite{Villani1} for a more precise description.  

\

L. Wu in \cite{Wu}, Mattingly, Stuart and Higham in \cite{Matt},  and Bakry, Cattiaux and Guillin in \cite{BCG} use the powerful method of Lyapunov functions to prove the exponential convergence to equilibrium.

\

All these approaches have in common to use global methods to prove the convergence to equilibrium in the sense that the functional inequalities that are used are written in an integrated form with respect to the invariant measure. In contrast,  our approach parallels the Bakry-Emery approach to hypercontractivity \cite{BE} and  is only  based on local computations: We just compute second order derivatives and use Cauchy-Schwarz inequality. To perform these computations the existence of an invariant measure is not even required. 

\

Let us now describe this local approach and what is the main idea of the present work. We can associate to $L=\sum_{i=1}^n X_i^2 +Y$ its \textit{carr\'e du champ} operator
 \[
\Gamma(f,g)=\frac{1}{2}\left(L(fg)- fLg-gLf\right)
\]
 and its iteration
 \[
 \Gamma_2(f,g)=\frac{1}{2}\left(L\Gamma(f,g)-\Gamma(f,Lg)-\Gamma(g,Lf)\right).
 \]

If the operator $L$ admits a symmetric measure $\mu$ and if for every $f$, $\Gamma_2(f,f)\ge \rho \Gamma(f,f)$ for some positive constant $\rho$, then it is known  the semigroup $P_t$ generated by $L$ will converge exponentially fast to an equilibrium (see \cite{BE}). However in a number of interesting situations including the ones described before, it is impossible to bound from below $\Gamma_2$ by $\Gamma$ alone.  Actually, the Bakry-Emery criterion $\Gamma_2 \ge \rho \Gamma$ requires some form of ellipticity of the operator $L$ and fails to hold for strictly subelliptic operators. In the recent few years, there have been several works, extending the Bakry-Emery approach to subelliptic diffusion operators. We mention in particular \cite{BB, BG, BW} where a generalized curvature dimension inequality is shown to be satisfied for a large class of geometrically relevant subelliptic diffusion operators. In particular, under suitable conditions, explicit rates of convergence to equilibrium are obtained for the semigroup. The hypoelliptic situations treated in these works are quite different from the ones we have in mind here, because in \cite{BB}, \cite{BG} or \cite{BW} the operator $\sum_{i=1}^n X_i^2$ is hypoelliptic and $Y$ is in the linear span of $X_1,\cdots,X_n$. Here, we  are interested in situations where $\sum_{i=1}^n X_i^2$ is fully degenerate. Since it is impossible to bound from below $\Gamma_2$ by $\Gamma$ alone, our main idea  will be to introduce a \textit{vertical} first order bilinear form $\Gamma^Z$ and to compute the curvature of $L$ in this new vertical direction:
 \[
 \Gamma^Z_2(f,g)=\frac{1}{2}\left(L\Gamma^Z(f,g)-\Gamma^Z(f,Lg)-\Gamma^Z(g,Lf)\right).
 \]
As it turns out, in the degenerate situations mentioned above, it is sometimes possible to find, under some conditions, such a form  $\Gamma^Z$ with the property that
\begin{align}\label{CD2}
\Gamma_2(f,f)+\Gamma^Z_2(f,f) \ge -K \Gamma(f,f)+ \rho \Gamma^Z(f,f),
\end{align}
where $K \in \mathbb{R}$ and where $\rho >0$. The important point here is that $\rho$ is positive and will hence induce a convergence to equilibrium in the missing vertical direction. The bound \eqref{CD2} is local and implies several interesting pointwise bounds for the semigroup $P_t$. In the case when there is an invariant probability measure, that satisfies the Poincar\'e inequality with respect to the new gradient $\Gamma+\Gamma^Z$, then it is proved that the inequality $\eqref{CD2}$ implies exponential convergence to equilibrium for the semigroup  in a Sobolev norm with an explicit rate. If the invariant measure satisfies the log-Sobolev inequality, we obtain then an entropic convergence to equilibrium.

\

Though the methods presented in the paper are general, for the sake of presentation, we focus on the case of the kinetic Fokker-Planck equation that was originally studied by Villani and point to the reference \cite{BauSur} for the presentation of a general framework. Another example of kinetic Fokker-Planck type equation where those methods apply is given in the recent work \cite{BT}.

\section{Convergence to equilibrium for the kinetic Fokker-Planck equation}

Let $V:\mathbb{R}^n \to \mathbb{R}$ be a smooth function. The kinetic Fokker-Planck equation with confinement potential $V$ is the parabolic  partial differential equation:
\begin{equation}\label{FP1}
\frac{\partial h}{\partial t}=\Delta_v h - v \cdot \nabla_v h+\nabla V \cdot \nabla_v h -v\cdot \nabla_x h , \quad (x,v) \in \mathbb{R}^{2n}.
\end{equation}
It is the Kolmogorov-Fokker-Planck equation associated to the stochastic differential system
\[
\begin{cases}
dx_t =v_t dt \\
dv_t=-v_t dt -\nabla V (x_t) dt +dB_t,
\end{cases}
\]
where $(B_t)_{ t\ge 0}$ is a Brownian motion in $\mathbb{R}^n$. The operator
\[
L=\Delta_v  - v \cdot \nabla_v +\nabla V \cdot \nabla_v -v\cdot \nabla_x 
\]
is not elliptic but it can be written in H\"ormander's form
\[
L=\sum_{i=1}^n X_i^2 +X_0+Y,
\]
where $X_i=\frac{\partial}{\partial v_i}$, $X_0=- v \cdot \nabla_v$ and $Y=\nabla V \cdot \nabla_v -v\cdot \nabla_x $. The vectors $$(X_1,\cdots,X_n, [Y,X_1],\cdots,[Y,X_n])$$  form a basis of $\mathbb{R}^{2n}$ at each point. This implies from H\"ormander's theorem that $L$ is hypoelliptic. The operator $L$ admits for invariant measure the measure
\[
d\mu=e^{-V(x)-\frac{\| v \|^2}{2}} dxdv.
\]
It is readily checked that $L$ is not symmetric with respect to $\mu$ but that the adjoint $L^*$ in  $L^2(\mu)$ is given by
\[
L^*=\sum_{i=1}^n X_i^2 +X_0-Y.
\]
The operator $L$ is the generator of a strongly continuous sub-Markov semigroup $(P_t)_{t \ge 0}$. If we assume that the Hessian $\nabla^2 V$ is bounded, then $P_t$ is Markovian (that is $P_t 1=1$) and for any bounded Borel function $f:\mathbb{R}^{2n} \to \mathbb{R}$, $(t,x,v) \to P_t f(x,v)$ is the unique solution of the Cauchy problem
\[
\begin{cases}
\frac{\partial h}{\partial t}=Lh \\
h(0,x,v)=f(x,v).
\end{cases}
\]

One of the main results of Villani (see also Helffer and Nier \cite{HN1} for related results) concerning the convergence to equilibrium of $P_t$ is the following theorem:

\begin{theorem}[Villani \cite{Villani1}, Theorem 35]\label{Vill}
Define $H^1(\mu)=\{ f \in L^2(\mu), \| \nabla f \| \in L^2(\mu)\}$.
Assume that there is a constant $c>0$ such that $\| \nabla^2 V \| \le c( 1+ \| \nabla V \|)$ and that the normalized invariant measure $d\mu=\frac{1}{Z}e^{-V(x)-\frac{\| v \|^2}{2}} dxdv$  is a probability measure that satisfies the classical Poincar\'e inequality
\[
\int_{\mathbb{R}^{2n}} \| \nabla f \|^2 d\mu \ge \kappa \left[ \int_{\mathbb{R}^{2n}} f^2 d\mu -\left( \int_{\mathbb{R}^{2n}} f d\mu\right)^2 \right].
\]
Then, there exist constants $C>0$ and $\lambda >0$ such that for every $f \in H^1(\mu)$, with $\int_{\mathbb{R}^{2n}} f d\mu=0$,
\[
\int_{\mathbb{R}^{2n}} (P_t f)^2 d\mu + \int_{\mathbb{R}^{2n}} \| \nabla P_t f \|^2 d\mu \le C e^{-\lambda t} \left( \int_{\mathbb{R}^{2n}} f^2 d\mu + \int_{\mathbb{R}^{2n}} \| \nabla  f \|^2 d\mu\right)
\]
\end{theorem}
It is worth observing that since $\mu$ is a product, it satisfies the Poincar\'e inequality as soon as the marginal measure $d\mu_x= e^{-V(x)} dx$ satisfies the Poincar\'e inequality on $\mathbb{R}^n$.

\

In this section, after some preliminaries, we first give in Section 2.3 a new proof of this result under the  assumption that the Hessian $\nabla^2 V$ be bounded. This is a stronger assumption than in Villani's result. However our method gives pointwise gradient estimates that can not be obtained by Villani's method.  It also provides a better control of the constants $C$ and $\lambda$.  As we shall see in Section 2.4, Villani's result can be recovered with its optimal assumption on $\nabla^2 V$  by integrating the local inequalities obtained in Section \ref{JKL}. In Section 2.5 we show that a very small variation of our method will almost immediately give an entropic convergence of $P_t$, under the assumption that $\mu$ satisfies the log-Sobolev inequality. This entropic convergence is also obtained by Villani under the assumption that $\nabla^2 V$ is bounded. In the final Section 2.6, we study the convergence in the Kantorovich-Wasserstein distance $W_2$. It is obtained under strong assumptions on the potential $V$ but has the advantage not to explicitly use the invariant measure.

\subsection{$\Gamma_2$ calculus for the kinetic Fokker-Planck equation}\label{JKL}

Throughout the section we assume that $\nabla^2 V$ is bounded.

 Following Bakry and \'Emery \cite{BE}  we associate to  $L=\Delta_v  - v \cdot \nabla_v +\nabla V \cdot \nabla_v -v\cdot \nabla_x$ the \text{carr\'e du champ} operator,
 \[
\Gamma(f,g)=\frac{1}{2}\left(L(fg)- fLg-gLf\right)=\nabla_v f \cdot \nabla_v g=\sum_{i=1}^n \frac{\partial f}{\partial v_i}\frac{\partial g}{\partial v_i}
\]
 and its iteration
 \[
 \Gamma_2(f,g)=\frac{1}{2}\left(L\Gamma(f,g)-\Gamma(f,Lg)-\Gamma(g,Lf)\right).
 \]
 For simplicity of notations, we will denote $\Gamma(f):=\Gamma(f,f)$ and $\Gamma_2(f):=\Gamma_2(f,f)$. A straightforward computation shows that:
 \begin{lemma}
 For $f \in C^\infty(\mathbb{R}^{2n})$,
 \begin{align*}
 \Gamma_2(f)& =\| \nabla_v^2 f \|^2 +\Gamma(f)+ \nabla_x f\cdot \nabla_v f, \\
  &=\sum_{i,j=1}^n \left( \frac{\partial^2 f}{\partial v_i \partial v_j}\right)^2 +\sum_{i=1}^n \left( \frac{\partial f}{\partial v_i} \right)^2+\sum_{i=1}^n  \frac{\partial f}{\partial x_i}  \frac{\partial f}{\partial v_i} 
  \end{align*}
 \end{lemma}
 
 The term $\nabla_x f\cdot \nabla_v f$ makes impossible to bound from below $\Gamma_2$ by $\Gamma$ alone. As a consequence the Bakry-\'Emery curvature of $L$ is $-\infty$. The idea is now  to introduce a carefully chosen vertical gradient and to compute the corresponding curvature of $L$ in this vertical direction. For $i=1,\cdots,n$, we denote $$Z_i= 2\frac{\partial }{\partial x_i }+\frac{\partial }{\partial v_i },$$
 and 
 \[
 Zf=2\nabla_x f +\nabla_v f.
 \]
 We define then 
 \[
 \Gamma^Z(f,g)= Zf \cdot Zg =\sum_{i=1}^n Z_i fZ_ig
 \]
 and
 \[
 \Gamma^Z_2(f,g)=\frac{1}{2}\left(L\Gamma^Z(f,g)-\Gamma^Z(f,Lg)-\Gamma^Z(g,Lf)\right).
 \]
 
 \begin{lemma}
 For $f \in C^\infty(\mathbb{R}^{2n})$,
 \begin{align*}
 \Gamma_2^Z(f)& =\| \nabla_v Z f\|^2 +\frac{1}{2}\Gamma^Z(f) +\frac{1}{2}\nabla_v f\cdot Z f-2\nabla^2V(\nabla_v f,Zf) \\
  &=\sum_{i,j=1}^n \left( \frac{\partial}{\partial v_i} Z_j f \right)^2+\frac{1}{2}\Gamma^Z(f)+\frac{1}{2}\sum_{i=1}^n\frac{\partial f}{\partial v_i}Z_if-2\sum_{i,j=1}^n \frac{\partial^2V}{\partial x_i \partial x_j}\frac{\partial f}{\partial v_i} Z_jf.
 \end{align*}
 \end{lemma}
 \begin{proof}
 Let us write 
 \[
L=\sum_{i=1}^n X_i^2 +X_0+Y,
\]
where $X_i=\frac{\partial}{\partial v_i}$, $X_0=- v \cdot \nabla_v$ and $Y=\nabla V \cdot \nabla_v -v\cdot \nabla_x $. We have
\begin{align*}
\Gamma^Z_2(f)&=\frac{1}{2}\left(L\Gamma^Z(f)-2\Gamma^Z(f,Lf)\right) \\
 &=\frac{1}{2}\left(L\left(\sum_{i=1}^n (Z_if)^2\right)-2\sum_{i=1}^n Z_if Z_iLf\right)\\
 &=\sum_{i=1}^n \Gamma(Z_i f) +\sum_{i=1}^n Z_if [L,Z_i]f \\
 &=\sum_{i=1}^n \Gamma(Z_i f) +\sum_{i=1}^n Z_if [X_0,Z_i]f +\sum_{i=1}^n Z_if [Y,Z_i]f.
\end{align*}
We now compute,
\[
[X_0,Z_i]=X_i
\]
and
\[
[Y,Z_i]=\frac{1}{2}Z_i-\frac{1}{2}X_i-2 \sum_{j=1}^n \frac{\partial^2V}{\partial x_i \partial x_j}\frac{\partial }{\partial v_j} .
\]
The result follows then easily.
 \end{proof}
 
 A consequence of the previous computations is the following lower bound for $\Gamma_2+\Gamma_2^Z$.
 
 \begin{proposition}\label{lowerbound}
 For every $0 < \eta <\frac{1}{2}$, there exists $K(\eta) \ge -\frac{1}{2}$ such that for every $f \in C^\infty(\mathbb{R}^{2n})$, 
 \[
 \Gamma_2(f)+\Gamma_2^Z(f) \ge -K(\eta) \Gamma(f)+\eta \Gamma^Z(f).
 \]
 \end{proposition}
 
 \begin{proof}
 As a consequence of the previous lemmas, we have
 \[
 \Gamma_2(f) \ge \frac{1}{2} \Gamma(f)+ \frac{1}{2}\nabla_v f\cdot Z f
 \]
 and
 \[
 \Gamma_2^Z(f) \ge \frac{1}{2}\Gamma^Z(f) +\frac{1}{2}\nabla_v f\cdot Z f-2\nabla^2V(\nabla_v f,Zf).
 \]
 We deduce
 \[
 \Gamma_2(f) + \Gamma_2^Z(f) \ge  \frac{1}{2} \Gamma(f)+\frac{1}{2}\Gamma^Z(f) +\nabla_v f\cdot Z f-2\nabla^2V(\nabla_v f,Zf).
 \]
 We now pick $0<\eta <\frac{1}{2}$ and $K \in \mathbb{R}$ and write
 \begin{align*}
  &  \frac{1}{2} \Gamma(f)+\frac{1}{2}\Gamma^Z(f) +\nabla_v f\cdot Z f-2\nabla^2V(\nabla_v f,Zf) \\
  =&\eta \Gamma^Z(f) -K\Gamma(f) +\left( \frac{1}{2}-\eta\right) \Gamma^Z(f)+\left(\frac{1}{2}+K\right)\Gamma(f) +(\mathbf{Id}-2\nabla^2V)(\nabla_v f,Zf).
 \end{align*}
 The bilinear form $\left( \frac{1}{2}-\eta\right) \Gamma^Z(f)+\left(\frac{1}{2}+K\right)\Gamma(f) +(\mathbf{Id}-2\nabla^2V)(\nabla_v f,Zf)$ can now be made non negative as soon as, in the sense of symmetric matrices,
 \[
 4 \left( \frac{1}{2}-\eta\right) \left(\frac{1}{2}+K\right) \ge ( \mathbf{Id}-2\nabla^2V )^2.
 \]
 The claim follows then from the fact that we assume that $\nabla^2V$ is bounded on $\mathbb{R}^n$
 \end{proof}
 
 \begin{remark}
 
 \
 
 \begin{itemize}
 \item The keypoint of the previous proposition is the positivity of $\eta$ which will imply the coercivity of  $P_t$ in the vertical direction. The sign of $K(\eta)$ is not that relevant in the sense that the coercivity of $P_t$ in the horizontal direction can  be obtained if we assume the invariant measure to satisfy a Poincar\'e inequality.
\item If there are constants $0<a<b<1$ such that for every $x \in \mathbb{R}^n$, $a \le \nabla^2 V \le b$, then the previous proof shows that we can chose $K(\eta)$ to be negative.
 \end{itemize}
 \end{remark}
 
 \subsection{Gradient bounds}
 Throughout the section we assume that $\nabla^2 V$ is bounded.
 
 We prove now some global  bounds for the gradient of the semigroup $(P_t)_{t \ge 0}$.  Related pointwise gradient bounds in this kinetic model were also obtained by Guillin and Wang \cite{FYW}, but our bounds have the advantage to be global. Let us also observe that such bounds can not be obtained by Villani's method.
 
 \
  
 The previous computations have shown that for $f \in C^\infty(\mathbb{R}^{2n})$, 
  \[
 \Gamma_2(f)+\Gamma_2^Z(f) \ge \lambda(\eta)( \Gamma(f)+ \Gamma^Z(f)).
 \]
 where $\lambda(\eta)= \min (-K(\eta),\eta)$. With this lower bound in hands we can use the methods introduced by Baudoin-Bonnefont \cite{BB} and F.Y. Wang \cite{FYW2} to obtain the following results.

 \begin{lemma}\label{GB}
 If $f$ is a bounded Lipschitz function on $\mathbb{R}^{2n}$, then for every $t \ge 0$, $P_tf$ is a bounded and  Lipschitz function. More precisely, with the notations of Lemma \ref{lowerbound}, for every $(x,v) \in \mathbb{R}^{2n}$,
  \[
  \Gamma(P_tf) (x,v)+\Gamma^Z(P_tf)(x,v) \le e^{-2\lambda(\eta)t} P_t ( \Gamma(f) +\Gamma^Z(f) )(x,v),
  \]
  where $\lambda(\eta)= \min (-K(\eta),\eta)$.
 \end{lemma}
 
 \begin{proof}
  The heuristic argument is the following:  We fix $(x,v) \in \mathbb{R}^{2n}$, $t >0$ and consider the functional
  \[
  \Psi(s)=P_s (\Gamma(P_{t-s} f) +\Gamma^Z(P_{t-s} f) )(x,v).
  \]
  Differentiating $\Psi$ , leads to
  \[
   \Psi'(s)=2P_s (\Gamma_2(P_{t-s} f) +\Gamma_2^Z(P_{t-s} f) )(x,v)\ge 2 \lambda(\eta) P_s (\Gamma(P_{t-s} f) +\Gamma^Z(P_{t-s} f) )(x,v)=2\lambda(\eta) \Psi(s).
  \]
  Therefore, we obtain $\Psi(t) \ge e^{\lambda(\eta) t} \Psi(0)$, which is the claimed inequality. In order to rigorously justify this argument, we observe that since $\nabla V$ is Lipschitz, the function $W(x,v)=1+\| x \|^2 +\| v\|^2$ is a Lyapunov function  such that, for some constant $C>0$, $LW \le C W$ and $\| \nabla W \| \le C W$. We can then use Proposition \ref{lowerbound} and argue like \cite{BB} Proposition 2.2, or \cite{FYW2}, Lemma 2.1. 
\end{proof}
 
 A direction computation shows that, since the vector fields $X_i$ and $Z_j$ commute, we do have the following intertwining of the quadratic forms $\Gamma$, $\Gamma^Z$: For every $f \in C^\infty(\mathbb{R}^{2n})$,
 \[
 \Gamma(f ,\Gamma^Z(f))=\Gamma^Z(f,\Gamma(f)).
 \]
 This intertwining leads to the following entropic type pointwise  bound:
 
\begin{lemma}\label{EB}
Let $f \in C^\infty(\mathbb{R}^{2n})$ be a positive function such that $\sqrt{f}$ is bounded and Lipschitz, then for $t \ge 0$, $\sqrt{P_t f}$ is bounded and Lipschitz. More precisely,  with the notations of Lemma \ref{lowerbound}, for every $(x,v) \in \mathbb{R}^{2n}$,
  \[
  P_t f (x,v)\Gamma(\ln P_tf) (x,v)+P_t f (x,v)\Gamma^Z(\ln  P_tf)(x,v) \le e^{-2\lambda(\eta)t} P_t (f \Gamma(\ln f) +f\Gamma^Z(\ln f) )(x,v),
  \]
  where $\lambda(\eta)= \min (-K(\eta),\eta)$.
    \end{lemma}

\begin{proof}
Again, we first give the heuristic argument. We consider the functional
  \[
  \Psi(s)=P_s ( P_{t-s} f \Gamma(\ln P_{t-s} f) +P_{t-s} f \Gamma^Z(\ln P_{t-s} f) )(x,v),
  \]
which by differentiation gives
  \[
  \Psi'(s)=2P_s (\ln P_{t-s} f \Gamma_2(\ln P_{t-s} f) +\ln P_{t-s} f \Gamma_2^Z(\ln P_{t-s} f) )(x,v)\ge 2 \lambda(\eta) \Psi(s).
  \]
Let us observe that the computation of the derivative crucially relies on the fact that  $ \Gamma(P_tf ,\Gamma^Z(\ln P_t f))=\Gamma^Z(P_t f,\Gamma(\ln P_t f))$. The rigorous justification can be given as above by using the Lyapunov function $W$ and Proposition 2.1 in \cite{BB}.
\end{proof}

 \subsection{Convergence in $H^1$}
 Throughout the section we assume that $\nabla^2 V$ is bounded.
 We are now in position to recover Villani's convergence result.
 
 \begin{theorem}\label{conve}
Assume that  the normalized  invariant measure $d\mu=\frac{1}{Z}e^{-V(x)-\frac{\| v \|^2}{2}} dxdv$ is a probability measure that satisfies the Poincar\'e inequality
\[
\int_{\mathbb{R}^{2n}}( \Gamma(f) +\Gamma^Z(f)) d\mu \ge \kappa \left[ \int_{\mathbb{R}^{2n}} f^2 d\mu -\left( \int_{\mathbb{R}^{2n}} f d\mu\right)^2 \right].
\]
With the notations of Lemma \ref{lowerbound}:
\begin{itemize}
\item If $K(\eta)+\eta >0$, then for every $f \in H^1(\mu)$, with $\int_{\mathbb{R}^{2n}} f d\mu=0$,
\begin{align*}
 & (\eta +K(\eta))\int_{\mathbb{R}^{2n}} (P_t f)^2d\mu +\int_{\mathbb{R}^{2n}} (\Gamma(P_tf) +\Gamma^Z(P_tf) )d\mu  \\
 \le &  e^{-\lambda t}\left(  (\eta +K(\eta)) \int_{\mathbb{R}^{2n}} f^2d\mu + \int_{\mathbb{R}^{2n}}(\Gamma(f) +\Gamma^Z(f)) d\mu \right) ,
\end{align*}
where $\lambda=\frac{2\eta \kappa}{\kappa+\eta +K(\eta)}$.
\item If $K(\eta)+\eta \le 0$, then for every $f \in H^1(\mu)$, with $\int_{\mathbb{R}^{2n}} f d\mu=0$,

$$
\begin{cases}
\int_{\mathbb{R}^{2n}} \Gamma(P_tf) +\Gamma^Z(P_tf) )d\mu \le e^{-2\eta t}  \int_{\mathbb{R}^{2n}} ( \Gamma(f) +\Gamma^Z(f)) d\mu \\
 \int_{\mathbb{R}^{2n}} (P_tf)^2 d\mu \le \frac{1}{\kappa} e^{-2\eta t}   \int_{\mathbb{R}^{2n}} ( \Gamma(f) +\Gamma^Z(f)) d\mu .
\end{cases}
$$
\end{itemize}
\end{theorem}
  \begin{remark}
  The two norms  $(\eta +K(\eta)) \int_{\mathbb{R}^{2n}} f^2d\mu + \int_{\mathbb{R}^{2n}}(\Gamma(f) +\Gamma^Z(f)) d\mu$ and $\int_{\mathbb{R}^{2n}} \| \nabla f \|^2 d\mu$ are equivalent on $H^1(\mu)$, so the previous result indeed implies Villani's theorem.
  \end{remark}
  \begin{proof}
  By a density argument we can and will prove these inequalities when $f$ is smooth, bounded and Lipschitz.
  
  Let us first assume that $K(\eta)+\eta >0$. We fix $t>0$ and consider the functional
  \[
  \Psi(s) =(K(\eta)+\eta)P_s( (P_{t-s} f)^2 ) + P_s( \Gamma(P_{t-s}f) +\Gamma^Z(P_{t-s}f) ).
  \]
  By repeating the arguments of the previous section, we get the differential inequality
  \[
  \Psi(s)-\Psi(0) \ge 2 \eta \int_0^s P_u(\Gamma(P_{t-u}f) +\Gamma^Z(P_{t-u}f) )du.
  \]
 Denote now $\varepsilon=\frac{\eta +K(\eta)}{\kappa +\eta +K(\eta)}$. We have  from the assumed Poincar\'e inequality
\[
\varepsilon\int_{\mathbb{R}^{2n}} \Gamma(P_{t-u}f) +\Gamma^Z(P_{t-u}f) d\mu \ge \varepsilon \kappa  \int_{\mathbb{R}^{2n}} (P_{t-u}f)^2 d\mu .
\]
Therefore, denoting $\Theta(s)=\int_{\mathbb{R}^{2n}} \Psi(s)d\mu$, we obtain
\begin{align*}
 \Theta(s)-\Theta(0)&  \ge 2 \eta(1-\varepsilon)\int_0^s  \int_{\mathbb{R}^{2n}} \Gamma(P_{t-u}f) +\Gamma^Z(P_{t-u}f) d\mu du+2\varepsilon \kappa \int_0^s  \int_{\mathbb{R}^{2n}} (P_{t-u}f)^2 d\mu du \\
  &\ge \lambda \int_0^s \Theta (u) du.
\end{align*}
We conclude with Gronwall's  inequality.

If $K(\eta)+\eta \le 0$, then the argument is identical by considering instead  the functional
 \[
  \Psi(s) =P_s (\Gamma(P_{t-s}f) +\Gamma^Z(P_{t-s}f) ).
  \]
  \end{proof}
  
  \subsection{Integrated $\Gamma_2$ calculus}
  
  Our goal is to prove Theorem \ref{Vill} under a weaker assumption than the boundedness of $\| \nabla^2 V \| $. We assume in this section that that there is a constant $c>0$ such that $$\| \nabla^2 V \| \le c( 1+ \| \nabla V \|).$$ The keypoint of this assumption is that it implies that there exists a constant $K>0$ such that for every $g \in H^1(e^{-V}dx)$,
    \begin{equation}\label{jkhgf}
  \int_{\mathbb{R}^n} \| \nabla^2 V \|^2 g ^2 e^{-V} dx \le K \left(  \int_{\mathbb{R}^n}  g ^2 e^{-V} dx+ \int_{\mathbb{R}^n}  \| \nabla g\| ^2 e^{-V}dx \right).
  \end{equation}
  We refer to Lemma A.18 in Villani's memoir \cite{Villani1}.
  
  \
  
  We need to adapt a little the method of Section 2.2 by adjusting the parameters. Let $\alpha,\beta \in \mathbb{R}$ with $\alpha \neq 0$, to be chosen later. For $i=1,\cdots,n$ we denote 
 $$Z_i= \alpha \frac{\partial }{\partial x_i }+\beta \frac{\partial }{\partial v_i },$$
 and 
 \[
 Zf=\alpha \nabla_x f +\beta \nabla_v f.
 \]
 As before, we define then 
 \[
 \Gamma^Z(f,g)= Zf \cdot Zg =\sum_{i=1}^n Z_i fZ_ig
 \]
 and
 \[
 \Gamma^Z_2(f,g)=\frac{1}{2}\left(L\Gamma^Z(f,g)-\Gamma^Z(f,Lg)-\Gamma^Z(g,Lf)\right).
 \]
 
 \begin{proposition}\label{integrated}
There exists a choice of $\alpha,\beta$ such that there exists $\eta >0$, and  $K(\eta) \in \mathbb{R}$ such that for every $f \in C_0^\infty(\mathbb{R}^{2n})$, 
 \[
 \int_{\mathbb{R}^{2n}} \Gamma_2(f)+\Gamma_2^Z(f) d\mu \ge \int_{\mathbb{R}^{2n}}   -K(\eta)\Gamma(f)+\eta \Gamma^Z(f) d\mu.
 \]
\end{proposition}
 
 \begin{proof}
 A direct computation shows that
 \[
 \Gamma_2^Z(f)=\| \nabla_v Zf \|^2 +\beta \left( 1-\frac{\beta}{\alpha} \right) \nabla_v f \cdot Zf + \frac{\beta}{\alpha} \| Zf \|^2 -\alpha \nabla^2V (\nabla_v f, Zf)
 \]
 From the inequality \eqref{jkhgf}, we can deduce that there exists a constant $C>0$ such that
 \[
 \int_{\mathbb{R}^{2n}} \nabla^2V (\nabla_v f, Zf)d\mu \ge -C \left(  \int_{\mathbb{R}^{2n}} \| \nabla^2_v f \|^2 +\| \nabla_v Zf\|^2 +\| \nabla_v f\|^2 + \|Zf\|^2 \right) d\mu.
 \]
 By using similar arguments as in the proof of Proposition \ref{lowerbound}, the proof is easily completed. 
 \end{proof}
 
 With Proposition \ref{integrated} in hands, the argument in the proof of Theorem \ref{conve} can be reproduced, therefore providing a new proof of Villani's result Theorem \ref{Vill}.

  \subsection{Entropic convergence}
  Throughout the section we assume that $\nabla^2 V$ is bounded.
  
  We  prove here the entropic convergence of $P_t$  to the equilibrium if the invariant measure satisfies log-Sobolev inequality. The result is obtained in a very similar way as the convergence in $H^1(\mu)$.
  
   \begin{theorem}\label{entropy}
Assume that  the normalized  invariant measure $d\mu=\frac{1}{Z}e^{-V(x)-\frac{\| v \|^2}{2}} dxdv$ is a probability that satisfies the log-Sobolev inequality
\[
\int_{\mathbb{R}^{2n}}(f \Gamma(\ln f) +f\Gamma^Z(\ln f)) d\mu \ge \kappa \left[ \int_{\mathbb{R}^{2n}} f \ln f d\mu -\left( \int_{\mathbb{R}^{2n}} f d\mu\right)\ln \left( \int_{\mathbb{R}^{2n}} f d\mu\right) \right].
\]
With the notations of Lemma \ref{lowerbound}:
\begin{itemize}
\item If $K(\eta)+\eta >0$, then for every positive and bounded $f \in C^\infty(\mathbb{R}^{2n})$, such that $\sqrt{f}$ is Lipschitz and $\int_{\mathbb{R}^{2n}} f d\mu=1$,
\begin{align*}
 & 2(\eta +K(\eta))\int_{\mathbb{R}^{2n}} P_t f \ln P_t f d\mu +\int_{\mathbb{R}^{2n}} (P_tf \Gamma(\ln P_tf) +P_tf \Gamma^Z(\ln P_tf) )d\mu  \\
 \le &  e^{-\lambda t}\left( 2 (\eta +K(\eta)) \int_{\mathbb{R}^{2n}} f \ln f d\mu + \int_{\mathbb{R}^{2n}}(f\Gamma(\ln f) +f\Gamma^Z(\ln f )) d\mu \right) ,
\end{align*}
where $\lambda=\frac{2\eta \kappa}{\kappa+2(\eta +K(\eta))}$.
\item If $K(\eta)+\eta \le 0$,  then for every positive and bounded $f \in C^\infty(\mathbb{R}^{2n})$, such that $\sqrt{f}$ is Lipschitz and $\int_{\mathbb{R}^{2n}} f d\mu=1$,

$$
\begin{cases}
\int_{\mathbb{R}^{2n}} P_t f \Gamma(\ln P_tf) +P_t f \Gamma^Z(\ln P_tf) )d\mu \le e^{-2\eta t}  \int_{\mathbb{R}^{2n}} (f  \Gamma(\ln f) +f \Gamma^Z(\ln f)) d\mu \\
 \int_{\mathbb{R}^{2n}} P_tf \ln P_t f  d\mu \le \frac{1}{\kappa} e^{-2\eta t}   \int_{\mathbb{R}^{2n}} ( f\Gamma(\ln f) +f\Gamma^Z(\ln f)) d\mu .
\end{cases}
$$
\end{itemize}
\end{theorem}

\begin{proof}
The proof is identical  to the proof of Theorem \ref{conve} by considering now the functional
  \[
  \Psi(s) =2 \max(K(\eta)+\eta, 0)P_s ( P_{t-s} f \ln P_{t-s} f) +P_s ( P_{t-s} f \Gamma(\ln P_{t-s} f) +P_{t-s} f \Gamma^Z(\ln P_{t-s} f) ).
  \]
  \end{proof}

  \subsection{Convergence in the Kantorovich-Wasserstein distance}
  
  To conclude, we prove a convergence to equilibrium in the Kantorovich-Wasserstein distance. The method is very similar to the methods developed earlier, but has the advantage not to assume anything on the invariant measure.
  
 We will denote $\mathcal{P}_2(\mathbb{R}^{2n})$ the set of probability measures on $\mathbb{R}^{2n}$ which have a finite second moment. We denote by $W_2$ the usual $L^2$ Kantorovich-Wasserstein distance associated to the Euclidean metric on $\mathbb{R}^{2n}$. It is given for $\mu , \nu \in \mathcal{P}_2(\mathbb{R}^{2n})$ by
 \[
 W_{2} (\mu,\nu)^2=\inf \mathbb{E}( \| X-Y \|^2)
 \]
where the infimum is taken over the set of random variables $X,Y$ such that $X \sim \mu$, $Y \sim \nu$.

\begin{theorem}\label{plk}
Assume that there exist constants $m,M>0$ such that $$m \le \nabla^2 V \le M$$ and $\sqrt{M} -\sqrt{m} \le 1$. Then, there exist constants $C_1,C_2>0$  such that for every $\mu,\nu \in \mathcal{P}_2(\mathbb{R}^{2n})$, and $t \ge 0$,
\[
W_{2} (P^*_t \mu , P^*_t \nu) \le C_1 e^{-C_2t} W_{2} (\mu,\nu).
\]
\end{theorem}

\begin{proof}
 Let $\alpha,\beta,\gamma,\delta \in \mathbb{R}$ be constants to be chosen later and consider the gradient
\[
\mathcal{T}(f)=\sum_{i=1}^{n} \left( \alpha \frac{\partial f}{\partial x_i} +\beta  \frac{\partial f}{\partial v_i}\right)^2 +\left( \gamma \frac{\partial f}{\partial x_i} +\delta  \frac{\partial f}{\partial v_i}\right)^2.
\]
We denote 
\[
\mathcal{T}_2 (f)=\frac{1}{2} ( L \mathcal{T}(f) - 2 \mathcal{T} (f , Lf)).
\]
In the first step of the proof, we prove that we can chose $\alpha,\beta,\gamma,\delta \in \mathbb{R}$ such that for some constant $\rho >0$, 
\begin{align}\label{con2}
\mathcal{T}_2 (f) \ge \rho \mathcal{T}(f),
\end{align}

\

We can write $\mathcal{T}$ in the form
\[
\mathcal{T}(f)=\sum_{i=1}^{2n} (Z_i f)^2
\]
with
\begin{align*}
Z_i=
\begin{cases}
\alpha \frac{\partial f}{\partial x_i} +\beta  \frac{\partial f}{\partial v_i} , 1 \le i \le n \\
 \gamma \frac{\partial f}{\partial x_{n-i}} +\delta  \frac{\partial f}{\partial v_{n-i}}, n+1 \le i \le 2n.
\end{cases}
\end{align*}

 and 
 \[
L=\sum_{i=1}^n X_i^2 +X_0+Y,
\]
where $X_i=\frac{\partial}{\partial v_i}$, $X_0=- v \cdot \nabla_v$ and $Y=\nabla V \cdot \nabla_v -v\cdot \nabla_x $.  We have then
\begin{align*}
\mathcal{T}_2 (f)&=\frac{1}{2} ( L \mathcal{T}(f) - 2 \mathcal{T} (f , Lf)) \\
 &=\frac{1}{2}\left(L\left(\sum_{i=1}^{2n} (Z_if)^2\right)-2\sum_{i=1}^{2n} Z_if Z_iLf\right)\\
 &=\sum_{i=1}^{2n}\sum_{j=1}^n (X_jZ_i f)^2 +\sum_{i=1}^{2n} Z_if [L,Z_i]f \\
 &=\sum_{i=1}^{2n} \sum_{j=1}^n (X_jZ_i f)^2+\sum_{i=1}^{2n} Z_if [X_0,Z_i]f +\sum_{i=1}^{2n} Z_if [Y,Z_i]f.
\end{align*}

As a consequence we obtain
\[
\mathcal{T}_2 (f) \ge \sum_{i=1}^{2n} Z_if [X_0,Z_i]f +\sum_{i=1}^{2n} Z_if [Y,Z_i]f.
\]

We now  compute
\begin{align*}
[X_0,Z_i]=
\begin{cases}
\beta  \frac{\partial }{\partial v_i} , 1 \le i \le n \\
\delta  \frac{\partial }{\partial v_{i-n}}, n+1 \le i \le 2n.
\end{cases}
\end{align*}
and 
\begin{align*}
[Y,Z_i]=
\begin{cases}
\beta  \frac{\partial }{\partial x_i}-\alpha\sum_{j=1}^n  \frac{\partial^2 V}{\partial x_i \partial x_j} \frac{\partial }{\partial v_i}  , 1 \le i \le n \\
\delta  \frac{\partial }{\partial x_{i-n}}-\gamma\sum_{j=1}^n  \frac{\partial^2 V}{\partial x_i \partial x_j} \frac{\partial }{\partial v_{i-n}} , n+1 \le i \le 2n.
\end{cases}
\end{align*}
As a consequence, we obtain after straightforward computations
\begin{align*}
\mathcal{T}_2 (f) \ge  &  (\beta^2+\delta^2) \| \nabla_v f \|^2+(\alpha \beta +\gamma \delta) \| \nabla_x f \|^2+ (\beta^2 +\delta^2 +\alpha \beta +\gamma \delta)  \nabla_v f  \cdot  \nabla_x f \\
 & -(\alpha^2+\gamma^2) \nabla^2 V ( \nabla_v f, \nabla_x f)-(\alpha \beta +\gamma \delta) \nabla^2 V ( \nabla_v f, \nabla_v f).
\end{align*}
The right-hand side of the above inequality can be seen as a bilinear form on $\mathbb{R}^{2n}$ applied to $\nabla f= ( \nabla_x f, \nabla_v f)$. We want this form to be definite positive.

\

We first chose  $\alpha,\beta,\gamma,\delta$ such that $\alpha^2 +\gamma^2 =1$ and denote
\[
a=\alpha \beta +\gamma \delta, \quad  b =\beta^2+\delta^2.
\]
Observe that the only constraint on $a,b$ is that $a>0$ and  $ a ^2 \le b$.
A sufficient condition for the bilinear form to be definite positive is that for any eigenvalue $\lambda$ of $\nabla^2 V$, we have
\[
(a+b-\lambda)^2 < 4a (b-a \lambda).
\]
This inequality is equivalent to
\[
-\sqrt{\kappa^2-\theta^2} +\kappa < \lambda < \sqrt{\kappa^2-\theta^2} +\kappa,
\]
where 
\[
\kappa=a+b -2a^2, \quad \theta=b-a.
\]
We thus want to  chose $\alpha, \beta,\gamma,\delta$ in such a way that
\[
-\sqrt{\kappa^2-\theta^2} +\kappa=m, \sqrt{\kappa^2-\theta^2} +\kappa=M.
\]
This condition is equivalent to
\begin{align*}
\begin{cases}
\kappa=\frac{m+M}{2} \\
\theta^2=mM
\end{cases}
\end{align*}
We finally conclude by observing that the system
\begin{align*}
\begin{cases}
a+b -2a^2=\frac{m+M}{2} \\
b-a=\sqrt{mM}
\end{cases}
\end{align*}
has a solution $0<a^2 \le b$ as soon as $\sqrt{M}-\sqrt{m} \le 1$. As a conclusion,  there exist $\alpha,\beta,\gamma,\delta \in \mathbb{R}$ such that for some constant $\rho >0$, we have for every smooth  $f$
\begin{align}\label{con3}
\mathcal{T}_2 (f) \ge \rho \mathcal{T}(f).
\end{align}
Similarly to the proof of Lemma \ref{GB}, one deduces that for every smooth and Lipschitz function $f$ one has:
\[
\mathcal{T}(P_t f) \le e^{-2\rho t} P_t \mathcal{T}(f).
\]
Kuwada's duality \cite{Kuwada} yields then the conclusion since the distance
\[
d(x,y)= \sup \{ | f(x)-f(y)| , \mathcal{T}(f) \le 1 \}
\]
is equivalent to the Euclidean distance.
\end{proof}

\end{document}